\documentclass[12pt, a4paper]{amsart}

\pdfoutput=1

\usepackage{enumerate}
\usepackage{mathtools}
\usepackage[nocompress]{cite}
\usepackage{newtxtext, newtxmath}
\usepackage{hyperref}
\usepackage[margin=3.14cm]{geometry}
\usepackage{xcolor}
\usepackage{tikz}
\usetikzlibrary{decorations.pathmorphing}
\usepackage{microtype}
\usepackage{subfigure}

\theoremstyle{plain}
\newtheorem{theorem}{Theorem}
\newtheorem*{theorem*}{Theorem}
\newtheorem{proposition}[theorem]{Proposition}
\newtheorem{lemma}[theorem]{Lemma}
\newtheorem{corollary}[theorem]{Corollary}

\theoremstyle{definition}
\newtheorem{definition}[theorem]{Definition}
\newtheorem{remark}[theorem]{Remark}
\newtheorem{example}[theorem]{Example}
\newtheorem{principle}[theorem]{Principle}

\newcommand		{\p}[1]	{\left(#1\right)}

\newcommand		{\abs}[1]{\left|#1\right|}

\newcommand    	{\R} {\mathbb R}

\newcommand 	{\F} {\mathbb F}

\newcommand    	{\Z} {\mathbb Z}

\newcommand 	{\im} {\mathrm{im}}
\newcommand 	{\mult} {\mathrm{mult}}
\newcommand 	{\homol} {\mathrm{H}}
\newcommand 	{\ve} {\varepsilon}
\newcommand 	{\cdim} {\dim_{\mathrm{c}}}

\newcommand 	{\andim} {\dim_{\mathrm{an}}}

\DeclareMathOperator*{\fin}{\Rightarrow}
\DeclareMathOperator*{\longfin}{\xRightarrow{\hspace*{1cm}}}

\title[Systems of Equations with Graph Structure]{Dimension Bounds for Systems of Equations with Graph Structure}
\author{Eddie Nijholt}
\author{Davide Sclosa}
\date{\today}

\begin{document}

\begin{abstract}
We introduce a broad class of equations that are described by a graph, which includes many well-studied systems.
For these, we show that the number of solutions (or the dimension of the solution set)
can be bounded by studying certain induced subgraphs.
As corollaries,
we obtain novel bounds in spectral graph theory on the multiplicities of graph eigenvalues,
and in nonlinear dynamical system on the dimension of the equilibrium set
of a network.
\end{abstract}

\maketitle


\section{Introduction} \label{sec:intro}
By Gaussian elimination, all solutions of a linear system of equations
can be found by assigning arbitrary values to certain free variables,
corresponding to the kernel, from which all the other variables are uniquely determined.
In particular, the dimension of the solution set is the minimum number of 
variables needed to uniquely determine the others.

The implicit function theorem extends this method to nonlinear systems, but it only
applies locally and under certain non-degeneracy conditions.
In fact, even for polynomial systems, determining the dimension of the solution set
is a non-trivial problem, especially if the field is not algebraically closed.

In this paper we show that if a system of equations is sufficiently structured,
then an upper bound on the dimension of the solution set can be obtained
in a purely combinatorial fashion. For this, we abandon the idea of functional dependence
and use relations instead.
Suppose a solution set has the property that, upon fixing the first $k$ coordinates,
the remaining coordinates can only take on finitely many values.
Then by any ``sensible'' definition of dimension, the solution set should have a most
dimension~$k$.

\begin{principle} \label{principle}
If a system of equations is sufficiently well described by a graph~$G$,
then the number of solutions, or the dimension of the solution set,
can be bounded from above by studying certain induced subgraphs of~$G$.
\end{principle}

Principle~\ref{principle}, which is deliberately vague, is the main
idea behind all results of the paper. Section~\ref{sec:def} makes things precise
by introducing the necessary definitions.
Section~\ref{sec:theorem} contains our main theorem,
which can be stated, for now informally, as follows.

\begin{theorem*} [Forest Bound, Informal]
Consider a system of equations in the variables~$(x_v)_{v\in V}$ that is ``sufficiently
compatible'' with a graph~$G=(V,E)$, in a way that will be made precise.
Let~$F\subseteq G$ be an induced forest.
Let~$L\subseteq F$ be obtained by choosing from each component of~$F$ all leaves except one.
If values~$x_v$ are fixed for the vertices in~$G\setminus F$ and in~$L$,
then there are only finitely many possible values for the remaining vertices.
In particular the solution set has dimension at most~$\abs{G}-\abs{F}+\abs{L}$.
\end{theorem*}

Informally,
the theorem says that if the graph~$G$ contains a large induced forest with not too many leafs,
then the dimension of the solution set is small.
In the example of Figure~\ref{fig:induced_subforest} we have~$\abs{G}=15$, $\abs{F}=13$,
and~$\abs{L}=4$, thus we obtain the upper bound~$6$.

In Section~\ref{sec:combinatorics} we give applications
to spectral graph theory. We show that graphs containing
large induced subforests (e.g. long induced paths, large induced complete trees)
must have small eigenvalue multiplicities, thus a large number of distinct eigenvalues.

In Section~\ref{sec:dynamics} we give applications
to dynamical systems.
The idea is as follows:
After a network has reached an equilibrium, knowing the state of a few nodes
can be sufficient to determine the state of the whole system.
We obtain two bounds on the set of equilibria,
one based on tree subgraphs, one based on cycle subgraphs.

\begin{figure}
\centering
\begin{tikzpicture}[scale=1.3]
\node[circle,draw=black, fill=lightgray, fill opacity = 1, inner sep=2pt, minimum size=12pt] (1v) at (0,0) {};
\node[circle,draw=black, fill=black, fill opacity = 1, inner sep=2pt, minimum size=12pt] (2v) at (0.7,-1) {};
\node[rectangle, draw=black, fill=black, fill opacity = 1, inner sep=2pt, minimum size=12pt] (3v) at (1,0.1) {};
\node[circle,draw=black, fill=black, fill opacity = 1, inner sep=2pt, minimum size=12pt] (4v) at (0.83,1) {};
\node[circle,draw=black, fill=black, fill opacity = 1, inner sep=2pt, minimum size=12pt] (5v) at (1.8,-1.1) {};
\node[circle,draw=black, fill=black, fill opacity = 1, inner sep=2pt, minimum size=12pt] (6v) at (1.9,-0.1) {};
\node[circle,draw=black, fill=black, fill opacity = 1, inner sep=2pt, minimum size=12pt] (7v) at (2,0.81) {};
\node[rectangle,draw=black, fill=black, fill opacity = 1, inner sep=2pt, minimum size=12pt] (8v) at (2.9,-0.91) {};
\node[circle,draw=black, fill=black, fill opacity = 1, inner sep=2pt, minimum size=12pt] (9v) at (2.8,0.2) {};
\node[rectangle,draw=black, fill=black, fill opacity = 1, inner sep=2pt, minimum size=12pt] (10v) at (3,1) {};
\node[rectangle,draw=black, fill=black, fill opacity = 1, inner sep=2pt, minimum size=12pt] (11v) at (4,-1) {};
\node[circle,draw=black, fill=lightgray, fill opacity = 1, inner sep=2pt, minimum size=12pt] (12v) at (3.91,-0.11) {};
\node[circle,draw=black, fill=black, fill opacity = 1, inner sep=2pt, minimum size=12pt] (13v) at (4.05,0.96) {};
\node[circle,draw=black, fill=black, fill opacity = 1, inner sep=2pt, minimum size=12pt] (14v) at (4.95,-0.49) {};
\node[circle,draw=black, fill=black, fill opacity = 1, inner sep=2pt, minimum size=12pt] (15v) at (5.07,0.51) {};
\draw [-,   thick, gray, shorten <=-1pt, shorten >=-1pt] (1v) to  (2v);
\draw [-,   thick, gray, shorten <=-1pt, shorten >=-1pt] (1v) to  (3v);
\draw [-,   thick, gray, shorten <=-1pt, shorten >=-1pt] (1v) to  (4v);
\draw [-,   thick, gray, shorten <=-1pt, shorten >=-1pt] (1v) to  (5v);
\draw [-,   thick, gray, shorten <=-1pt, shorten >=-1pt] (12v) to  (8v);
\draw [-,   thick, gray, shorten <=-1pt, shorten >=-1pt] (12v) to  (9v);
\draw [-,   thick, gray, shorten <=-1pt, shorten >=-1pt] (12v) to  (11v);
\draw [-,   thick, gray, shorten <=-1pt, shorten >=-1pt] (12v) to  (13v);
\draw [-,   thick, gray, shorten <=-1pt, shorten >=-1pt] (12v) to  (14v);
\draw [-,   thick, gray, shorten <=-1pt, shorten >=-1pt] (12v) to  (15v);
\draw [-,   line width=1.5pt,  black, shorten <=-1pt, shorten >=-1pt] (5v) to  (2v);
\draw [-,   line width=1.5pt, black, shorten <=-1pt, shorten >=-1pt] (6v) to  (3v);
\draw [-,   line width=1.5pt, black, shorten <=-1pt, shorten >=-1pt] (7v) to  (4v);
\draw [-,   line width=1.5pt, black, shorten <=-1pt, shorten >=-1pt] (9v) to  (7v);
\draw [-,   line width=1.5pt, black, shorten <=-1pt, shorten >=-1pt] (8v) to  (5v);
\draw [-,   line width=1.5pt, black, shorten <=-1pt, shorten >=-1pt] (6v) to  (9v);
\draw [-,   line width=1.5pt, black, shorten <=-1pt, shorten >=-1pt] (9v) to  (13v);
\draw [-,   line width=1.5pt, black, shorten <=-1pt, shorten >=-1pt] (9v) to  (10v);
\draw [-,   line width=1.5pt, black, shorten <=-1pt, shorten >=-1pt] (13v) to  (15v);
\draw [-,   line width=1.5pt, black, shorten <=-1pt, shorten >=-1pt] (15v) to  (14v);
\draw [-,   line width=1.5pt, black, shorten <=-1pt, shorten >=-1pt] (11v) to  (14v);
\node[circle,draw=black, fill=lightgray, fill opacity = 1, inner sep=2pt, minimum size=12pt] (1v2) at (0,0) {};
\node[circle,draw=black, fill=black, fill opacity = 1, inner sep=2pt, minimum size=12pt] (2v2) at (0.7,-1) {};
\node[rectangle, draw=black, fill=black, fill opacity = 1, inner sep=2pt, minimum size=12pt] (3v2) at (1,0.1) {};
\node[circle,draw=black, fill=black, fill opacity = 1, inner sep=2pt, minimum size=12pt] (4v2) at (0.83,1) {};
\node[circle,draw=black, fill=black, fill opacity = 1, inner sep=2pt, minimum size=12pt] (5v2) at (1.8,-1.1) {};
\node[circle,draw=black, fill=black, fill opacity = 1, inner sep=2pt, minimum size=12pt] (6v2) at (1.9,-0.1) {};
\node[circle,draw=black, fill=black, fill opacity = 1, inner sep=2pt, minimum size=12pt] (7v2) at (2,0.81) {};
\node[rectangle,draw=black, fill=black, fill opacity = 1, inner sep=2pt, minimum size=12pt] (8v2) at (2.9,-0.91) {};
\node[circle,draw=black, fill=black, fill opacity = 1, inner sep=2pt, minimum size=12pt] (9v2) at (2.8,0.2) {};
\node[rectangle,draw=black, fill=black, fill opacity = 1, inner sep=2pt, minimum size=12pt] (10v2) at (3,1) {};
\node[rectangle,draw=black, fill=black, fill opacity = 1, inner sep=2pt, minimum size=12pt] (11v2) at (4,-1) {};
\node[circle,draw=black, fill=lightgray, fill opacity = 1, inner sep=2pt, minimum size=12pt] (12v2) at (3.91,-0.11) {};
\node[circle,draw=black, fill=black, fill opacity = 1, inner sep=2pt, minimum size=12pt] (13v2) at (4.05,0.96) {};
\node[circle,draw=black, fill=black, fill opacity = 1, inner sep=2pt, minimum size=12pt] (14v2) at (4.95,-0.49) {};
\node[circle,draw=black, fill=black, fill opacity = 1, inner sep=2pt, minimum size=12pt] (15v2) at (5.07,0.51) {};
\end{tikzpicture}
\caption{An induced subforest (black) of a graph (gray) and a set of selected
leaves (squares) obtained by
choosing from each component of the forest all leaves except one.
}
\label{fig:induced_subforest}
\end{figure}
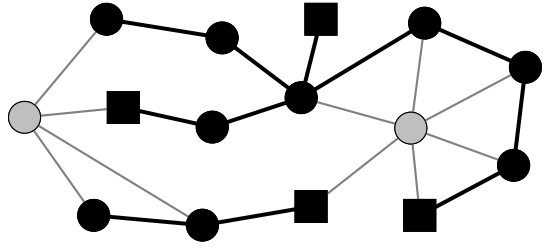


\section{Finite Determinacy, Compatibility, and Combinatorial Dimension} \label{sec:def}

Deriving a theorem from Principle~\ref{principle} requires introducing several definitions.
These definitions are the object of this section.

\subsection{Finite Determinacy}
In what follows~$V$ is a finite set of indexes.
Associated to each index~$v\in V$, there is a set~$X_v$ and a variable~$x_v$.
The variables~$(x_v)_{v\in V}$ will be used to represent statements about the sets~$(X_v)_{v\in V}$
as short formal expressions, which can be manipulated algebraically.
In applications, the variables~$(x_v)_{v\in V}$ will be those appearing
in a system of equations and~$X\subseteq \prod_{v\in V} X_v$ will be the solution set.

\begin{definition} \label{def:fin_det}
Let~$(X_v)_{v\in V}$ be a collection of sets,
let~$(x_v)_{v\in V}$ be a collection of variables,
and let~$X \subseteq \prod_{v\in V} X_v$ be any subset.
Given two subsets~$A,B\subseteq V$, we say that~$x_B = (x_v)_{v\in B}$
is $d$-\emph{determined} by~$x_A = (x_v)_{v\in A}$, and write
\[
	x_A \fin^d x_B,
\]
if for every element~$(\overline x_a)_{a\in A} \in \prod_{a\in A} X_a$
there are at most $d$~elements~$(\overline x_b)_{b\in B} \in \prod_{b\in B} X_b$
such that~$(\overline x_v)_{v\in V} \in X$
for some~$(\overline x_{u})_{u\in V\setminus (A\cup B)}
	\in \prod_{u\in V\setminus (A\cup B)} X_u$.
We say that~$x_B$ is \emph{finitely determined} by~$x_A$, and write~$x_A \fin x_B$,
if~$x_A \fin^d x_B$ for some~$d\geq 0$.
\end{definition}

Note that~$\fin$ depends on the set~$X$ and the inclusion~$X\subseteq \prod_{v\in V} X_v$.
The following example might clarify the definition.

\begin{example} \label{ex:torus}
Let $X$ be a two-dimensional torus, seen as a subset of $\R^4$ via
\begin{equation}
	X = \{(x_1,x_2,x_3,x_4) \in \R^4 \mid x_1^2+x_2^2 = x_3^2 + x_4^2= 1\}\, .
\end{equation}
Then $x_{\{1,3\}} \fin^4 x_{\{2,4\}}$, as for all fixed values of $x_1$ and $x_3$ we are left with at most two options for $x_2$ and $x_4$:
\begin{equation}
	x_2 = \pm \sqrt{1 - x_1^2}\, \, \text{ and } \, x_4 = \pm \sqrt{1 - x_3^2}\, ,
\end{equation}
with no options if $x_1^2>1$ and $x_3^2>1$, respectively. 
Note that we do not have~$x_{\{1,2\}} \fin x_{\{3,4\}}$, as there are in general still infinitely many possible values for $(x_3,x_4)$ when $x_1$ and $x_2$ are known.
\end{example}

\begin{remark}
In Definition~\ref{def:fin_det}, instead of writing $x_A \fin x_B$ we could have simply used the notation $A \fin B$. 
After all, the condition of being finitely determined by a subset of variables is ultimately just a statement about the index set~$V$, whenever the set~$X\subseteq \prod_{v\in V} X_v$ is fixed.
We have chosen for the former notation, however, as it allows us to make a distinction between relations in different sets. For instance, we can denote relations of finite determinacy in $X \subseteq \prod_{v\in V} X_v$ and $Y \subseteq \prod_{v\in V} Y_v$ by $x_A \fin x_B$ and $y_A \fin y_B$, respectively.
\end{remark}

\begin{remark}
An equivalent way of formulating Definition~\ref{def:fin_det} uses the projection maps 
\[
	\pi_A \colon \prod_{v\in V} X_v \to \prod_{v \in A}X_v, \quad
	(x_v)_{v\in V} \mapsto (x_v)_{v\in A},
\]
defined for all $A \subseteq V$. One verifies that $x_A \fin x_B$ if and only if for all $c \in \prod_{v \in A}X_v$ the set
\begin{equation}\label{projectionmapeq}
	\pi_B(\pi_A^{-1}(c) \cap X)
\end{equation}
is finite (with a uniform bound on the cardinality when ranging over $c$), and~$x_A \fin^d x_B$ if and only if the set has at most~$d$ elements.
We may generalize our definition of finite determinacy by dropping uniformity in~$c$.
Most statements in this paper remain true with this generalization, though we are not aware of any practical benefits.
\end{remark}

Finite determinacy has the following elementary properties.
We will refer to these properties by name or use them implicitly.

\begin{lemma}
Let~$X \subseteq \prod_{v\in V} X_v$ and~$A,B,C\subseteq V$. The following facts are true:
\begin{enumerate} [(i)]
\item (idempotence) $x_A \fin^1 x_A$
\item (monotonicity) if~$x_A \fin^n x_C$ and~$A\subset B$ then~$x_B \fin^n x_C$
\item (conjunction) if~$x_A \fin^n x_B$ and~$x_A \fin^{m} x_C$ then~$x_A \fin^{nm} x_{B\cup C}$
\item (transitivity) if~$x_A \fin^n x_B$ and~$x_B \fin^{m} x_C$ then~$x_A \fin^{nm} x_C$
\item (substitution) if~$x_A \fin^n x_B$ and~$x_C\fin^{m} x_D$,
then~$x_{(C\setminus B)\cup A} \fin^{nm} x_D$.
\end{enumerate}
\end{lemma}
\begin{proof}
Idempotence and monotonicity are trivial. Conjunction and transitivity follow from a simple
counting argument. Let us prove substitution.
By idempotence and monotonicity we have~$x_{(C\setminus B)\cup A} \fin^{1} x_{C\setminus B}$.
By hypothesis and monotonicity we have~$x_{(C\setminus B)\cup A} \fin^{n} x_B$.
Thus by conjunction
we have~$x_{(C\setminus B)\cup A} \fin^{n} x_{(C\setminus B)\cup B} = x_{C\cup B}$.
By hypothesis and monotonicity it follows that~$x_{C\cup B} \fin^m x_D$,
and thus by transitivity~$x_{(C\setminus B)\cup A} \fin^{nm} x_D$.
\end{proof}

We will often state results in the form~$x_{A} \fin^d x_{V\setminus A}$
for some subset~$A\subseteq V$. Note that this is equivalent to~$x_{A} \fin^d x_{V}$.

It will be convenient to introduce the following notation.

\begin{definition}
Let~$A,B,C\subseteq V$. We say that~$x_B$ is $d$-\emph{determined} by~$A$
relative to~$C$, and write~$x_A\fin^d_C x_B$, if~$x_{A\cup C}\fin^d x_B$.
\end{definition}

\subsection{Equations Compatible with a Graph}

We use standard terminology from graph theory.
A \emph{graph} is a pair of finite sets~$G=(V,E)$
such that the elements of~$E$ are subsets of~$V$ of size~$2$. We require that~$V$ is not empty.
We denote~$\abs{G}=\abs{V}$ the number of vertices in~$G$.
For every vertex~$v\in V$ we denote~$N(v)$ the neighborhood of~$v$ and~$\deg v=\abs{N(v)}$.
In our convention, we therefore always have $v \notin N(v)$.
By abuse of notation, especially when more than one graph is involved,
we will write~$v\in G$ instead of~$v\in V$.

In order to derive a theorem from Principle~\ref{principle},
we need to specify the meaning of ``well described by a graph''.
In the case of linear systems, there is a natural way to do it:
We might say that a linear system, described by a matrix~$M$, is \emph{linearly compatible}
with a graph~$G$ if~$M_{ij}\neq 0$ if and only if~$(i,j)$ is an edge of~$G$.
This is essentially the approach in~\cite{johnson2006eigenvalues}.
The natural generalization of this approach to nonlinear systems would
be through the implicit function theorem. However, the implicit function theorem
requires regularity, some non-degeneracy, and only applies locally.
Our approach avoids these issued entirely by defining compatibility
in terms of solution sets rather than equations.

\begin{definition} \label{def:compatible}
Let~$X \subseteq \prod_{v\in V} X_v$ be a subset and~$G=(V,E)$ be a graph.
We say that~$X$ and~$G$ are $d$-\emph{compatible} if
for every~$v\in V$ and every~$w\in N(v)$ we have
\[
		x_{(N(v)\setminus \{w\}) \cup \{v\}} \fin^d x_w.
\]
We say that~$X$ and~$G$ are \emph{strongly} $d$-\emph{compatible} if
moreover~$x_{N(v)}\fin^d x_v$ holds. We say that~$X$ and~$G$ are (strongly) compatible if they
are (strongly) $d$-compatible for some~$d\geq 0$.
\end{definition}

The notion of strong compatibility might appear more natural than compatibility,
and it does lead to simpler proofs and improved bounds. However, as we will see,
there are interesting systems in combinatorics
and dynamical systems that are naturally compatible with a graph,
but not strongly compatible with it.

Our first example comes from spectral graph theory. Definitions of adjacency matrix,
Laplacian, and normalized Laplacian,
can be found in~\cite{godsil2001algebraic, chung1997spectral}.

\begin{example} [Graph Spectrum] \label{ex:spectrum}
Let~$L$ be the Laplacian of a
graph~$G=(V,E)$ and fix~$\lambda\in \R$.
Consider the (possibly empty) eigenspace~$X = \{x\in \R^V \mid Lx = \lambda x\}$.
Let~$x \in X$. Then for every vertex~$v$
\[
	\sum_{w\in N(v)} x_w = (\deg v - \lambda) x_v.
\]
Therefore~$x_{(N(v)\setminus \{w\}) \cup \{v\}} \fin^1 x_w$
for every~$v\in V$ and every~$w\in N(v)$.
In particular~$X$ and~$G$ are $1$-compatible.
Moreover, unless~$\lambda = \deg v$ and~$\deg v$ is an eigenvalue for some~$v\in V$,
we have that~$X$ and~$G$ are $1$-strongly-compatible.
Similarly, eigenspaces of normalized Laplacian (resp. adjacency matrix) are
compatible with~$G$, and they are $1$-strongly-compatible
unless~$\lambda=1$ and~$1$ is an eigenvalue (resp.~$\lambda=0$ and~$0$ is an eigenvalue).
\end{example}

Our second example comes from dynamical systems.

\begin{example} [Network Dynamics] \label{ex:kuramoto}
Fix a graph~$G=(V,E)$. Let~$\Gamma$ be either~$\R$ or~$\R/2\pi \Z$.
For each edge~$e\in E$ let~$f_e:\Gamma\to \R$ be an odd function such
that~$\abs{f_e^{-1}(y)}\leq d<\infty$ for every~$y\in \R$, for some absolute constant~$d$.
Fix~$\omega \in \R^V$.
For every vertex~$v\in V$ let~$x_v\in \Gamma$ evolve according to the vector field
\begin{equation} \label{eq:network_dynamics}
	\dot x_v = \omega_v + \sum_{w\in N(v)} f_{\{w,v\}}(x_w-x_v).
\end{equation}
Then the set of equilibria~$X = \{x\in \Gamma^V : \dot x = 0\}$
is $d$-compatible with~$G$. In general~$X$ is not strongly compatible
(nilpotent equilibria on Eulerian graphs are counterexamples to
strong compatibility~\cite{sclosa2024completely}).
The particular case~$\Gamma=\R/2\pi \Z$ with~$f_{e}(x)=K \sin(x)$ and~$K>0$
is prototypical model of synchronization phenomena,
cf.~\cite{ashwin2016identical, Wiley2006, Ling2019, lu2020,
	Jafarian2018, Chen2019, acebron2005kuramoto, strogatz2000kuramoto};
in this case~$d=2$.
The particular case~$\Gamma = \R$ with~$\omega_v=0$ and~$f_{e}(x)=p(x)$
a non-constant odd polynomial, models polarization
in opinion dynamics,
cf.~\cite{devriendt2021nonlinear, homs2021nonlinear, srivastava2011bifurcations};
in this case~$d$ equals the degree of the polynomial~$p$.
\end{example}

\begin{definition}
Let~$G$ be a graph. Let~$H$ be a subgraph of~$G$ and~$C$ a subset of vertices.
Let~$X\subseteq \prod_{v\in G} X_v$.
We say that~$X$ and~$H$ are $d$-\emph{compatible}~\emph{relative to~$C$}
if for every~$v\in H$ and every~$w\in N_H(v)$
\[
	x_{(N_H(v)\setminus \{w\})\cup \{v\}} \fin^d_{C} x_w,
\]
where~$N_H(v)$ denotes the neighborhood of~$v$ in $H$.
Equivalently,~$X$ and~$H$ are $d$-compatible relative to~$H$
if for every~$c\in \prod_{v\in C} X_v$ the set~$X'\subseteq \prod_{v\in H} X_v$
defined by~$X' = \pi_H(X\cap \pi_C^{-1}(c))$ is $d$-compatible with~$H$.
Similarly, we say that$~X$ and$~H$ are strongly $d$-compatible relative to $C$ if they are compatible relative to $C$, and if in addition~$x_{N_H(v)} \fin^d_{C} x_v$ for all~$v\in H$.
\end{definition}

Recall that an induced subgraph of a graph~$G$ is any subgraph$~H$ obtained by first choosing any subset of vertices $W$ of~$G$, and then adding all edges of~$G$ that connect two vertices in $W$.
In contrast, a subgraph of~$G$ is defined as a subset of vertices, together with any subset of edges between these vertices. 
For instance, there may be many subgraphs of~$G$ consisting of all vertices, but only one such induced subgraph (the graph~$G$ itself).

Intuitively, the following result shows that an induced subgraph~$H$ of a
compatible graph~$G$ has the following property:
once values of~$(x_v)_{v\in G\setminus H}$
have been fixed, the system of equations in the remaining variables~$(x_v)_{H}$
(with~$(x_v)_{G\setminus H}$ considered as constants), is compatible with~$H$.

\begin{lemma} [Relative Compatibility of Induced Subgraphs] \label{lem:relative_induced}
Suppose that~$X\subseteq \prod_{v\in V(G)} X_v$ is $d$-compatible with~$G$.
Let~$H\subset G$ be an induced subgraph. Then~$X$ is $d$-compatible with~$H$
relative to~$G\setminus H$. The same result holds if compatible is replaced
by strongly compatible.
\end{lemma}
\begin{proof}
Let~$v\in H$ and~$w\in N_H(v)$. By hypothesis
\[
	x_{(N_G(v)\setminus \{w\}) \cup \{v\}} \fin^d x_w.
\]
Since$~H$ is induced, the set~$N_G(v)$ is the union of the sets~$N_G(v)\cap H = N_H(v)$
and~$N_G(v)\setminus H \subseteq G\setminus H$.
Therefore~$x_{(N_H(v)\setminus \{w\}) \cup (G\setminus H) \cup \{v\}} \fin x_w$,
which is the same as
\[
	x_{(N_H(v)\setminus \{w\}) \cup \{v\}} \fin_{G\setminus H}^d x_w.
\]
A similar argument applies to strong compatibility.
\end{proof}


\subsection{Combinatorial dimension}

Suppose that~$x_A\fin x_V$ for some subset~$A\subseteq V$.
Then, upon fixing~$\abs{A}$ coordinates, the remaining coordinates can only take
on finitely many values.
Therefore, by any ``sensible'' definition of dimension, the set~$X$ should have at most
dimension~$\abs{A}$. In this section, we introduce a notion of combinatorial dimension
for subsets of general set cartesian products.
Then we discuss its relation to other notions of dimension in
the case in which~$X$ has additional algebraic or topological structure.

\begin{definition} \label{def:dimension}
The \emph{combinatorial dimension} of~$X \subseteq \prod_{v\in V} X_v$, denoted by~$\cdim(X)$,
is the minimum cardinality of a subset~$A\subseteq V$ such that~$x_A \fin x_V$,
that is, it is the minimum number of coordinates needed to finitely determine all the others.
\end{definition}

Note that the notion of combinatorial dimension~$\cdim(X)$
is not an intrinsic property of the set~$X$, since it
depends on the inclusion~$X\subseteq \prod_{v\in V} X_v$, and
does not require a topology.

A subset~$X \subseteq \prod_{v\in V} X_v$ has combinatorial dimension~$0$ is and only if
it is finite. It has combinatorial dimension~$k$ if and only if there is a subset~$A\subseteq V$
of size~$\abs{A}=k$ such that~$x_A\fin x_V$ and no subset~$B\subseteq V$ of size~$\abs{B}<k$
has that property.

\begin{example}
For example, the set~$X$ in Example~\ref{ex:torus} has combinatorial dimension~$2$
since~$x_{\{1,3\}}\fin x_{\{1,2,3,4\}}$ and there is no~$i\in \{1,2,3,4\}$
such that~$x_i \fin x_{\{1,2,3,4\}}$.
\end{example}

The following proposition shows that for linear spaces over infinite fields,
combinatorial dimension and linear dimension are the same.
In particular, if~$\lambda\in \R$ is an eigenvalue of a graph~$G=(V,E)$
with multiplicity~$\mult(\lambda)$ as in Example~\ref{ex:spectrum}, then
the corresponding eigenspace~$X\subseteq \R^V$ satisfies~$\cdim(X) = \mult(\lambda)$.

\begin{proposition} [Linear Dimension] \label{prop:linear_dimension}
Let~$\F$ be an infinite field and~$X\subseteq \F^V$ a linear subspace.
Then
\[
	\dim_{\F} (X) = \cdim(X),
\]
where~$\dim_{\F}$ denotes the dimension of~$X$ as a linear subspace of~$\F^V$.
\end{proposition}
\begin{proof}
Suppose that~$x_A\fin x_V$ for some~$A\subseteq V$.
Then the canonical projection~$X\to \F^A$ has finite kernel. Since~$\F$ is infinite,
it follows that the kernel is~$\{0\}$, thus~$X\to \F^A$ is injective,
and in particular~$\dim_{\F}(X) \leq \abs{A}$. This proves~$\dim_{\F}(X) \leq \cdim(X)$.

Conversely, if~$\dim_{\F} (X)=k$, then any matrix with rows a linear basis of~$X$
contains~$k$ independent columns, say with indexes~$a_1,\ldots,a_k$,
and the set~$A=\{a_1,\ldots,a_k\}$ satisfies~$x_A\fin^1 x_V$.
This proves~$\cdim(X)\leq \dim_{\F} (X)$.
\end{proof}

A \emph{real-analytic set}~$X\subseteq \R^V$ is the solution set of a system
of real-analytic equations~\cite{massey2007notes, whitney1972complex, narasimhan2006introduction}.
Every real-analytic set is a locally-finite union of smooth manifolds, and
the \emph{analytic dimension}~$\andim(X)$ of~$X$ is the largest dimension of such a manifold.
The following remark shows that analytic dimension
is bounded from above by combinatorial dimension (but can be strictly smaller).
In particular, for the real-analytic set~$X$ of Example~\ref{ex:kuramoto}
we have~$\andim(X)\leq \cdim(X)$ (for simplicity, we state the result for~$\Gamma=\R$,
but the case~$\Gamma=\R/2\pi\Z$ is similar).

\begin{proposition} [Analytic Dimension] \label{prop:analytic_dimension}
Let~$X\subseteq \R^V$ be a real-analytic set. Then
\[
	\andim(X)\leq \cdim(X).
\]
Moreover, inequality can be strict.
\end{proposition}

\begin{proof}
Let~$A\subseteq V$ with~$\abs{A}=\cdim(X)$ such that~$x_A\fin x_V$,
that is, the canonical projection~$\pi: X \to \R^A$ has finite fibers.
Let~$M\subseteq X$ a manifold of dimension~$\andim(M) = \andim(X)$.
The restriction~$\pi: M \to \R^A$ is a smooth map between smooth manifolds.
By lower semi-continuity of rank, up to restricting to the open subset of~$M$ in which
rank is maximum, we can suppose that the differential of~$\pi: M\to \R^A$ has
constant rank~$r\geq 0$.
Clearly~$r\leq \abs{A}$. By the Constant Rank Level Set Theorem, each level set
of~$\pi: M\to \R^A$ is a closed submanifold of codimension~$r$. By hypothesis
level sets are finite, so~$r=\andim M$. We conclude that
\[
	\andim(X) = \andim M = r \leq \abs{A} = \cdim(X).
\]
To see that the inequality can be strict, note that
the real-analytic set~$X = \{(x_1,x_2)\in \R^2 : x_1 x_2=0\}$ satisfies~$\andim(X)=1$
but~$\cdim(X)=2$.
\end{proof}

The same proof shows that Proposition~\eqref{prop:analytic_dimension} remains true
if~$\R$ is replaced by~$\R/2\pi \Z$. More generally, if~$X\subseteq \R^V$
is replaced by~$X\subseteq M^V$, where~$M$ is a real-analytic manifold,
the conclusion becomes
\[
	\andim(X)\leq \andim(M)\cdot\cdim(X).
\]



\section{Main Theorem} \label{sec:theorem}
An induced forest~$F\subseteq G$ is an induced subgraph of~$G$
that is a forest as a graph on its own. The set of isolated
vertices of~$F$, that is, the set of vertices~$v\in F$ such that~$N_F(v)=\emptyset$
is an independent set of~$G$.

The goal of this section is proving the following theorem.


\begin{theorem} [Forest Bound] \label{thm:forest_bound}
Suppose that~$X\subseteq \prod_{v\in V} X_v$ and~$G=(V,E)$ are $d$-compatible.
Let~$F\subseteq G$ be an induced forest with no isolated vertices.
Let~$L\subseteq F$ be obtained by choosing from each component of~$F$
all leaves except one. Then
\begin{equation} \label{eq:forest_fin}
	x_{L \cup (G\setminus F)} \longfin^{d^{\abs{F\setminus L}}} x_{F\setminus L}
\end{equation}
and in particular the following bound holds:
\begin{equation} \label{eq:forest_bound}
	\cdim(X) \leq \abs{G} - \abs{F} + \abs{L}.
\end{equation}
If $X$~and~$G$ are strongly $d$-compatible,
the hypothesis that~$F$ has no isolated vertices can be dropped.
\end{theorem}

Note that~$L \cup (G\setminus F) = G\setminus (F\setminus L)$, thus~\eqref{eq:forest_fin}
is the same as~$x_{L \cup (G\setminus F)} \fin^{d^{\abs{F\setminus L}}} x_{G}$.
In particular~\eqref{eq:forest_bound} follows immediately from~\eqref{eq:forest_fin}.

By taking~$G=F$ in Theorem~\ref{thm:forest_bound} we obtain the
following apparently weaker theorem.

\begin{theorem} [Pure Forest Bound] \label{thm:pure_forest_bound}
Suppose that~$X\subseteq \prod_{v\in V} X_v$ and~$G=(V,E)$ are $d$-compatible.
Suppose that~$G$ is a forest with no isolated vertices.
Let~$L\subseteq G$ be obtained by choosing from each component of~$G$
all leaves except one. Then
\[
	x_{L} \longfin^{d^{\abs{G\setminus L}}} x_{G\setminus L}
\]
and in particular the following bound holds:
\begin{equation} \label{eq:pure_forest_bound}
	\cdim(X) \leq \abs{L}.
\end{equation}
If~$X$ and~$G$ are strongly $d$-compatible, the hypothesis that~$G$ has no isolated
vertices can be dropped.
\end{theorem}

Theorem~\ref{thm:forest_bound} implies Theorem~\ref{thm:pure_forest_bound}
by taking~$G=F$. We now show that, conversely, Theorem~\ref{thm:pure_forest_bound} implies
the apparently stronger Theorem~\ref{thm:forest_bound}.

\begin{proof}
[Proof that Theorem~\ref{thm:pure_forest_bound} implies Theorem~\ref{thm:forest_bound}]
Suppose that~$X$ is $d$-compatible with~$G$. Since~$F$ is an induced subgraph,
Lemma~\ref{lem:relative_induced} implies that~$X$ is $d$-compatible with~$F$
relative to~$G\setminus F$. By definition, this means that
for every fixed~$c\in \prod_{v\in G\setminus F} X_v$
the set~$X'\subseteq \prod_{v\in F} X_v$
defined by~$X' = \pi_F(X\cap \pi_{G\setminus F}^{-1}(c))$ is $d$-compatible with~$F$.
By Theorem~\ref{thm:pure_forest_bound} applied to the pair~$(X',F)$
we obtain~$x_L \fin^{d^{\abs{F\setminus L}}}_{(G\setminus F)} x_{F\setminus L}$,
which is the same as~$x_{L\cup (G\setminus F)} \fin^{d^{\abs{F\setminus L}}} x_{F\setminus L}$.
The same argument applies to strong compatibility.
\end{proof}

Before proceeding with the proof, let us discuss the content of these results.
Suppose that~$X$ and~$G$ are compatible.
Informally, Theorem~\ref{thm:forest_bound} states that if~$G$ contains a large induced
forest with not too many leaves, then the combinatorial dimension of~$X$ is small.
Let~$F\subseteq G$ be an induced forest with $l(F)$~leaves, $c(F)$~components, and
no isolated vertices. Then~\eqref{eq:forest_bound} is the same as
\begin{equation} \label{eq:weak_explicit}
	\cdim(X) \leq \abs{G} - \abs{F} + l(F) - c(F).
\end{equation}
In particular, for every induced subtree~$T$ with at least~$2$ vertices we have
\[
	\cdim(X) \leq \abs{G} - \abs{T} + l(T) - 1.
\]
In the particular case where~$G=F$ itself is a forest with no isolated vertices,
we obtain~$\cdim(X) < l(F)$: The combinatorial dimension of~$X$ is strictly smaller
than the number of leaves. In particular~$\cdim(X) < l(T)$ if~$G=T$ is a tree with at least
$2$~vertices.

Suppose now that~$X$ and~$G$ are strongly compatible, so that
induced forests are allowed to have isolated vertices.
Let~$(F\cup Z)\subseteq G$ be an induced forest with set of isolated vertices~$Z$.
From~\eqref{eq:forest_bound} we obtain the stronger upper bound
\begin{equation} \label{eq:strong_explicit}
	\cdim(X) \leq \abs{G} - \abs{F} + l(F) - c(F) - \abs{Z},
\end{equation}
compare~\eqref{eq:weak_explicit} with~\eqref{eq:strong_explicit}.
In particular, taking~$F=\emptyset$ gives
\[
	\cdim(X) \leq \abs{G} - \alpha(G)
\]
where~$\alpha(G)$ is the maximum size of an independent set in~$G$.
As we will see, for the stronger bound~\eqref{eq:strong_explicit} to hold, strong compatibility is
required.

In the proof of Theorem~\ref{thm:forest_bound}, we will consider rooted
trees, which are trees with a distinguished vertex.
There is a one-to-one correspondence between
rooted trees and tree orders, defined as follows.

\begin{definition}
Let~$T$ be a finite set.
A \emph{tree order} is a partial order~$(T,\prec)$
such that for every~$v\in W$
the set~$T_{\prec v} = \{w\in T \mid w \prec v\}$ is well-ordered,
and such that there is exactly one minimal element.
If~$v\prec w$ and there is no~$u$
such that~$v\prec u \prec w$, we call~$v$ the \emph{immediate predecessor} of~$w$
and~$w$ an \emph{immediate successor} of~$v$.
\end{definition}

We will also use the notation~$T_{\succ v} = \{w\in T \mid v\prec w\}$.
The following lemma is the key ingredient in the proof of Theorem~\ref{thm:forest_bound}.

\begin{lemma} [Tree Lemma] \label{lem:tree}
Let~$T$ be a finite set and~$X \subseteq \prod_{v\in T} X_v$.
Let~$\prec$ be a tree order on~$T$ and let~$M$ denote the set of maximal elements in~$(T,\prec)$.
Suppose that for every~$v\in T\setminus M$ and for every immediate successor~$w$ of~$v$
\begin{equation} \label{eq:tree}
	x_{\{w\} \cup T_{\succ w}} \fin^d x_v.
\end{equation}
Then~$x_{M} \fin^{d^{\abs{T\setminus M}}} x_{T\setminus M}$.
\end{lemma}
\begin{proof}
For every~$v\in T$ let~$\mathrm{dist}(v,M)$ denote the smallest integer~$m\geq 0$
such that there is a sequence of immediate successors~$v_0\prec v_1 \prec \cdots \prec v_m$
in~$T$ with~$v_0=v$ and~$v_m\in M$. We reason by induction on~$\mathrm{dist}(v,M)$.

Choose~$\overline x_{M} = (\overline x_{u})_{u\in M}$ arbitrarily.
Consider an element~$v\in T$ with~$\mathrm{dist}(v,M)=0$. Then~$v\in M$,
thus~$\overline x_v$ has already been chosen.
Consider an element~$v\in T$ such that~$\mathrm{dist}(v,M)=1$. Then~$v$ has an immediate
successor~$w\in M$. Since~$T_{\succ w}$ is empty, from~\eqref{eq:tree}
it follows that there are at most~$d$ many choices for~$\overline x_v$.
Make such a choice for each individual~$v\in T$ with~$\mathrm{dist}(v,M)=1$.

Now suppose that a choice~$\overline x_u \in X_u$
has been made for all~$u\in T$ with~$\mathrm{dist}(u,M)\leq n$.
Consider an element~$v\in T$ such that~$\mathrm{dist}(v,M)=n+1$. Then~$v$ has an
immediate successor~$w\succ v$.
Since every~$u\in \{w\}\cup T_{\succ w}$ satisfies~$\mathrm{dist}(u,M)\leq n$,
a choice has already been made for~$\overline x_{\{w\} \cup T_{\succ w}}$,
and from~\eqref{eq:tree} it follows that there are at most~$d$ many choices for~$\overline x_v$.

Once this procedure terminates, a choice has been made
for each~$v\in T\setminus M$ without repetitions.
It follows that for every fixed~$\overline x_{M}$
there are at most~$d^{\abs{T\setminus M}}$ choices for~$\overline x_{T\setminus M}$.
\end{proof}

We observed that Theorem~\ref{thm:forest_bound} and Theorem~\ref{thm:pure_forest_bound}
are equivalent. Let us prove the latter.

\begin{proof} [Proof of Theorem~\ref{thm:forest_bound}]
Let~$X$, $G$, and~$L$ be as in the statement.
Let~$T$ be a component of~$G$. Then~$T$ is a tree. The set~$T\cap L$ contains
all leaves of~$T$ except one. Let~$(T,\prec)$ denote the unique tree order
with minimum element the excluded leaf, and set of maximal elements~$M:=T\cap L$.
Choose a vertex $v \in T \setminus L$ and let $w \in T$ be an immediate successor of $v$.
Note that~$N(w) \subseteq \{v\} \cup T_{\succ w}$.
Since~$X$ and~$G$ are $d$-compatible,
we have~$x_{(N(w)\setminus \{v\}) \cup \{w\}} \fin^d x_v$,
thus by monotonicity~$x_{\{w\}\cup T_{\succ w}} \fin^d x_v$.
Lemma~\ref{lem:tree} implies
\[
	x_{T\cap L} \fin^{d^{\abs{T\setminus L}}}
		x_{T\setminus L}.
\]
Since~$T\cap L \subseteq L$, by monotonicity it follows
that~$x_{L} \fin^{d^{\abs{T\setminus L}}} x_{T\setminus L}$.
Now let~$T$ vary among the components of~$G$. Note that~$L = \bigcup_{T} (L\cap T)$
and that~$\sum_{T} \abs{T\setminus L} = \abs{G\setminus L}$.
Therefore, by the conjunction property it follows that
\[
	x_{L} \longfin^{\prod_T d^{\abs{T\setminus L}}}
		x_{\bigcup_T (T\setminus L)} = x_{G\setminus L},
\]
where~$\prod_T d^{\abs{T\setminus L}} = d^{\abs{G\setminus L}}$.
In particular, this implies~$\cdim(X) \leq \abs{L}$,
concluding the proof in the case of compatibility.
The argument in the case of strong compatibility is similar. It is enough do address
isolated vertices.
Let~$v\in G$ be an isolated vertex.
Since~$X$ is $d$-strongly compatible with~$G$, we have~$x_{N(v)} \fin^d x_v$.
Since~$N(v)=\emptyset$, this means that there are only $d$~possibilities for~$x_v$.
\end{proof}


\section{Applications to Spectral Graph Theory} \label{sec:combinatorics}
In this section we apply Theorem~\ref{thm:forest_bound} to spectral graph theory.
We obtain bounds
which relate eigenvalue multiplicities to induced
subgraphs. First, we state the corollary in full generality, then we discuss particular cases.

\begin{corollary} [Forest Bound for Graph Spectrum] \label{cor:spectrum_I}
Let~$\lambda$ be an eigenvalue of the Laplacian
(resp. normalized Laplacian, resp. adjacency matrix)
of a graph~$G$ and let~$\mult(\lambda)$ denote its multiplicity.
Let~$F$ be an induced forest in~$G$ with $l(F)$~leaves, $c(F)$~components and
no isolated vertices. Then
\begin{equation} \label{cor:forest_bound_spectrum_1}
	\mult(\lambda) \leq \abs{G} - \abs{F} + l(F) - c(F).
\end{equation}
In particular, if~$T$ is an induced tree with at least~$2$ vertices,
then~$\mult(\lambda) \leq \abs{G} - \abs{T} + l(T) - 1$.
Moreover, if~$\lambda\notin \{\deg v : v\in V\}$
(resp. $\lambda\neq 1$ for the normalized Laplacian,
resp. $\lambda \neq 0$ for the adjacency matrix)
and~$F \cup Z$ is an induced forest with set of isolated vertices~$Z$, then
\begin{equation} \label{cor:forest_bound_spectrum_2}
	\mult(\lambda) \leq \abs{G} - \abs{F} + l(F) - c(F) - \abs{Z}.
\end{equation}
In particular~$\mult(\lambda) \leq \abs{G}-\alpha(G)$,
where~$\alpha(G)$ denotes the largest size of an independent set of vertices.
\end{corollary}
\begin{proof}
In Example~\ref{ex:spectrum} we showed that the eigenspace~$X$ associated to an eigenvalue~$\lambda$
is compatible with~$G$,
and in particular strongly compatible if~$\lambda\notin \{\deg v : v\in V\}$
in the case of a Laplacian, with similar statements for the other matrices.
Moreover, by Proposition~\ref{prop:linear_dimension} we have~$\cdim(X)=\mult(\lambda)$.
Therefore,
the upper bounds~\eqref{cor:forest_bound_spectrum_1} and~\eqref{cor:forest_bound_spectrum_2}
follow immediately from Theorem~\ref{thm:forest_bound}. The last statement follows
by noting that an an independent set is the same as an induced forest without edges.
\end{proof}

Let us discuss some consequences.
The following result shows that graphs
containing large induced forests with not too many leaves must have many distinct eigenvalues.

\begin{corollary} [Number of Distinct Eigenvalues]
Let~$0<\ve<1$.
Let~$G$ be a graph and~$F\subseteq G$ an induced forest without isolated vertices such that
\[
	\abs{F}-l(F)+c(F)\geq \ve \abs{G}
\]
Let~$\ell$ denote the number of distinct Laplacian eigenvalues of~$G$. Then
\[
	\ell \geq \frac{1}{1-\ve}.
\]
The statement remains true if ``Laplacian'' is replaced by ``normalized Laplacian''
or by ``adjacency matrix''.
\end{corollary}
\begin{proof}
By Corollary~\ref{cor:spectrum_I}, for every eigenvalue~$\lambda$
we have~$\mult(\lambda) \leq (1-\ve) \abs{G}$. Therefore
\[
	\abs{G} = \sum_{\lambda\ \text{distinct}} \mult(\lambda) \leq \ell (1-\ve) \abs{G}.
\]
Dividing by~$\abs{G}$ completes the proof.
\end{proof}

The next result shows that graphs with an eigenvalue of large multiplicity
cannot contain long induced paths.

\begin{corollary} [Spectral Bound on Longest Induced Path]
Let~$k\geq 1$ be the maximum length of an induced path in~$G$. Then
\[
	k \leq 1 + \abs{G}- \max_{\lambda}\mult(\lambda)
\]
where~$\lambda$ ranges among Laplacian, normalized Laplacian, and adjacency eigenvalues of~$G$.
The upper bound is reached, for example, if~$G$ has at least one vertex and no edges.
\end{corollary}
\begin{proof}
Let~$P_k\subseteq G$ be an induced path of length~$k$. Corollary~\ref{cor:spectrum_I}
implies~$k-1 = \abs{P_k}-l(P_k)+1 \leq \abs{G}-\mult(\lambda)$.
\end{proof}

The \emph{complete binary tree}~$T_{2,k}$ of \emph{depth}~$k\geq 1$
is defined inductively as follows: $T_{2,1}$ is a graph with one vertex,
and~$T_{2,k+1}$ is obtained by attaching $2$~vertices to each leaf of~$T_{2,k}$.
In particular~$T_{2,k}$ has~$2^{k}-1$ vertices and~$2^{k-1}$ leaves.
The next result shows that graphs with an eigenvalue of large multiplicity
cannot contain large induced complete trees.

\begin{corollary} [Spectral Bound on Largest Induced Complete Tree]
Let~$k\geq 1$ be the maximum depth of an induced complete binary tree in~$G$. Then
\[
	k \leq 1 + \log_2 \p{\abs{G}-\max_{\lambda} \mult(\lambda)}
\]
where~$\lambda$ ranges among Laplacian, normalized Laplacian, and adjacency eigenvalues of~$G$.
The upper bound is reached, for example, if~$G$ is a complete binary tree of depth~$2$.
\end{corollary}
\begin{proof}
Let~$T_{2,k}\subseteq G$ be an induced complete binary subtree of depth~$k$.
Corollary~\ref{cor:spectrum_I}
implies~$2^{k-1} = \abs{T_{2,k}}-l(T_{2,k})+1 \leq \abs{G}-\mult(\lambda)$.
\end{proof}

A similar upper bound can be obtained for $r$-ary trees with~$r>2$.
The next result is a spectral bound on~$\alpha(G)$, the size of the largest independent set in~$G$.

\begin{corollary} [Spectral Bound on Largest Independent Set] \label{cor:independent_set}
Let~$\lambda \notin \{\deg v: v\in V\}$ be an eigenvalue of the Laplacian of~$G$,
or let~$\lambda\neq 1$ be an eigenvalue of the normalized Laplacian of~$G$,
or let~$\lambda\neq 0$ be an eigenvalue of the adjacency matrix of~$G$. Then
\[
	\alpha(G) \leq \abs{G} - \mult(\lambda).
\]
\end{corollary}
\begin{proof}
Follows immediately from Corollary~\ref{cor:spectrum_I}.
\end{proof}

The following example shows that the extra condition on~$\lambda$
in Corollary~\ref{cor:independent_set} is necessary.
In particular, it shows that for the stronger bound to hold
in Theorem~\ref{thm:forest_bound}, strong compatibility is in general necessary.

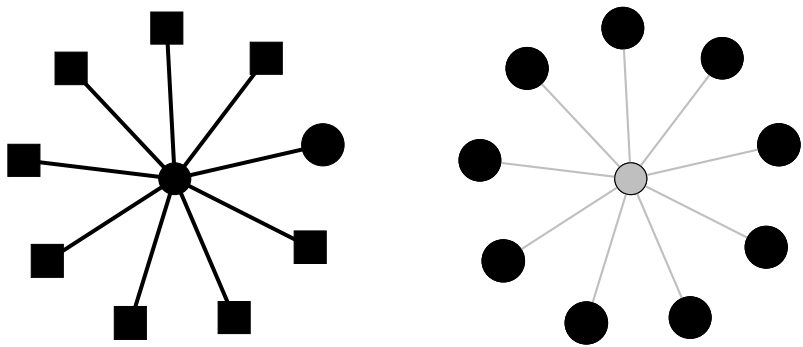
\begin{figure}
\centering
\begin{tikzpicture}[scale=1]
	\node[circle,draw=black, fill=black, fill opacity = 1, inner sep=2pt, minimum size=12pt] (m) at (0,0) {};
	\node[circle,draw=black, fill=black, fill opacity = 1, inner sep=2pt, minimum size=12pt](9v) at 		({2*cos(360+13)}, {2*sin(360 + 13)}) {9};

	\foreach \i in {1,...,8}{
		\node[rectangle,draw=black, fill=black, fill opacity = 1, inner sep=2pt, minimum size=12pt](\i v) at ({2*cos(360/9*\i+13)}, {2*sin(360/9*\i + 13)}) {\i};
		\draw [-,   line width=1.5pt,  black, shorten <=-1pt, shorten >=-1pt] (m) to  (\i v);}
		\draw [-,   line width=1.5pt,  black, shorten <=-1pt, shorten >=-1pt] (m) to  (9v);

\node[circle,draw=black, fill=gray, fill opacity = 0.3, inner sep=2pt, minimum size=12pt] (m2) at (0+6,0) {};
\foreach \i in {1,...,9}{
		\node[circle,draw=black, fill=black, fill opacity = 1, inner sep=2pt, minimum size=12pt](\i v2) at ({2*cos(360/9*\i+13)+6}, {2*sin(360/9*\i + 13)}) {\i};
		\draw [-,   thick,  lightgray, shorten <=-1pt, shorten >=-1pt] (m2) to  (\i v2);}
\node[circle,draw=black, fill=lightgray, fill opacity = 1, inner sep=2pt, minimum size=12pt] (m2) at (0+6,0) {};
\foreach \i in {1,...,9}{
		\node[circle,draw=black, fill=black, fill opacity = 1, inner sep=2pt, minimum size=12pt](\i v3) at ({2*cos(360/9*\i+13)+6}, {2*sin(360/9*\i + 13)}) {\i};}
\end{tikzpicture}
\caption{The star graph with~$n=10$ vertices, together with two examples of an induced subforest (black) with selected leaves (squares). The subforest on the left gives a bound of~$8$ on the multiplicity of any eigenvalue (in general,~$n-2$). The subforest on the right gives a bound of $1$, but does not apply to all eigenvalues due to isolated vertices.}
\label{fig:star_graph}
\end{figure}

\begin{example} \label{ex:star}
Let~$G$ be a star graph with~$n>3$ vertices, thus~$n-1$ leaves,
see Figure \ref{fig:star_graph}. Clearly~$\alpha(G)=n-1$.
Its Laplacian eigenvalues are~$0$ and $n$ with multiplicity~$1$,
and~$1$ with multiplicity~$n-2$. Note that~$1\in \{\deg v: v\in V\}$
and that~$\mult(1)+\alpha(G)=(n-2)+(n-1)>n$.
A similar argument applies to the normalized Laplacian and to the adjacency matrix.
\end{example}

If the graph~$G$ is a tree, Corollary~\ref{cor:spectrum_I} reduces
to a known bound~\cite[Theorem 2.3]{grone1990laplacian}~\cite{smith1970some}:

\begin{corollary} [Multiplicities in Trees]
Let~$T$ be a tree with at least~$2$ vertices. Let~$\lambda$ be a Laplacian eigenvalue.
Then~$\mult(\lambda) \leq l(T)-1$.
The statement remains true if ``Laplacian'' is replaced by ``normalized Laplacian''
or by ``adjacency matrix''.
\end{corollary}
\begin{proof}
Since~$T$ has at least one edge, we can take~$T=F$ in~\eqref{cor:forest_bound_spectrum_1}.
\end{proof}


\section{Applications to Dynamical Systems} \label{sec:dynamics}
This section applies Theorem~\ref{thm:forest_bound} to the dynamical
systems of Example~\ref{ex:kuramoto}, which we briefly reintroduce for convenience
of the reader.
Fix a graph~$G=(V,E)$. Let~$\Gamma$ be either~$\R$ or~$\R/2\pi \Z$.
For each edge~$e\in E$ let~$f_e:\Gamma\to \R$ be an odd function such
that~$\abs{f_e^{-1}(y)}\leq d<\infty$ for every~$y\in \R$, where~$d$ is an absolute constant.
Fix~$\omega \in \R^V$.
For every vertex~$v\in V$ let~$x_v\in \Gamma$ evolve according to the vector field
\begin{equation} \label{eq:network_dynamics_bis}
	\dot x_v = \omega_v + \sum_{w\in N(v)} f_{\{w,v\}}(x_w-x_v).
\end{equation}
Then the set of equilibria~$X = \{x\in \Gamma^V : \dot x = 0\}$
is $d$-compatible with~$G$.
Recall that~$\andim(X)$ denotes the analytic dimension of~$X$,
that is, the maximum dimension of a manifold contained in~$X$
(the set~$X$ might contain singularities, e.g. two manifolds of equilibria might intersect,
but singularities are nowhere dense in~$X$).

\begin{corollary} [Forest Bound for Network Dynamics] \label{cor:kuramoto}
Let~$X \subseteq \Gamma^V$ be the set of equilibria
of the vector field~\eqref{eq:network_dynamics_bis} on a given graph~$G=(V,E)$.
Let~$F$ be an induced forest in~$G$ without isolated vertices. Then the analytic dimension of~$X$
satisfies
\[
	\andim(X) \leq \abs{G} - \abs{F} + l(F) - c(F).
\]
In particular, the right hand side is an upper bound to the maximum dimension of a
manifold of equilibria.
\end{corollary}
\begin{proof}
By Proposition~\ref{prop:analytic_dimension}, it is enough to prove the upper bound
for the combinatorial dimension~$\cdim(X)$.
Example~\ref{ex:kuramoto} shows that
the equilibrium set~$X$ is compatible with~$G$.
Theorem~\eqref{thm:forest_bound} applies.
\end{proof}

In the remainder of this section, we obtain a second upper bound on~$\andim(X)$
based on ``forests of cycle subgraphs'' rather than on ``forests of vertices''.
To properly introduce the result, we need to recall some basic notions
from graph homology~\cite{diestel2024graph, hatcher2005algebraic}.

Fix, once for all, an arbitrary orientation of each edge~$e\in E$.
Let~$B \in \{+1,-1,0\}^{V\times E}$ be the incidence matrix of~$G$ for the given
orientation, defined by
\[
	B_{v, (w,u)} =
	\begin{dcases*}
		1, & if $v=u$ \\
		-1, & if $v=w$ \\
		0, & otherwise.
	\end{dcases*}
\]
Then the vector field~\eqref{eq:network_dynamics_bis} can be rewritten compactly as
\[
	\dot x = \omega - B f(B^t x),
\]
where~$f(y_{e_1},\ldots,y_{e_{\abs{E}}})
	= (f_{e_1}(y_{e_1}),\ldots,f_{e_{\abs{E}}}(y_{e_{\abs{E}}}))$,
and can be decomposed into maps
\[
	B^t: \Gamma^V\to \Gamma^E, \quad f:\Gamma^E \to \R^E, \quad B: \R^E \to \R^V.
\]
The kernel~$\homol_0(G) := \ker B^t \subseteq \Gamma^V$, known as \emph{zeroth homology group},
has dimension~$c(G)$ and contains the vectors
that are constant in each connected component.
Every oriented cycle in~$G$ (with no repeated edges)
can be encoded as a vector~$c_i$ in~$\R^E$ such that~$c_{i,e}=+1,-1,0$ if~$c_i$
traverses the oriented~$e$ in its orientation, the opposite orientation,
or does not contain~$e$ respectively.
Then the kernel~$\homol_1(G) := \ker B \subseteq \R^E$,
known as \emph{first homology group},
is generated by~$\abs{E}-\abs{V}+c(G)$ oriented cycles.

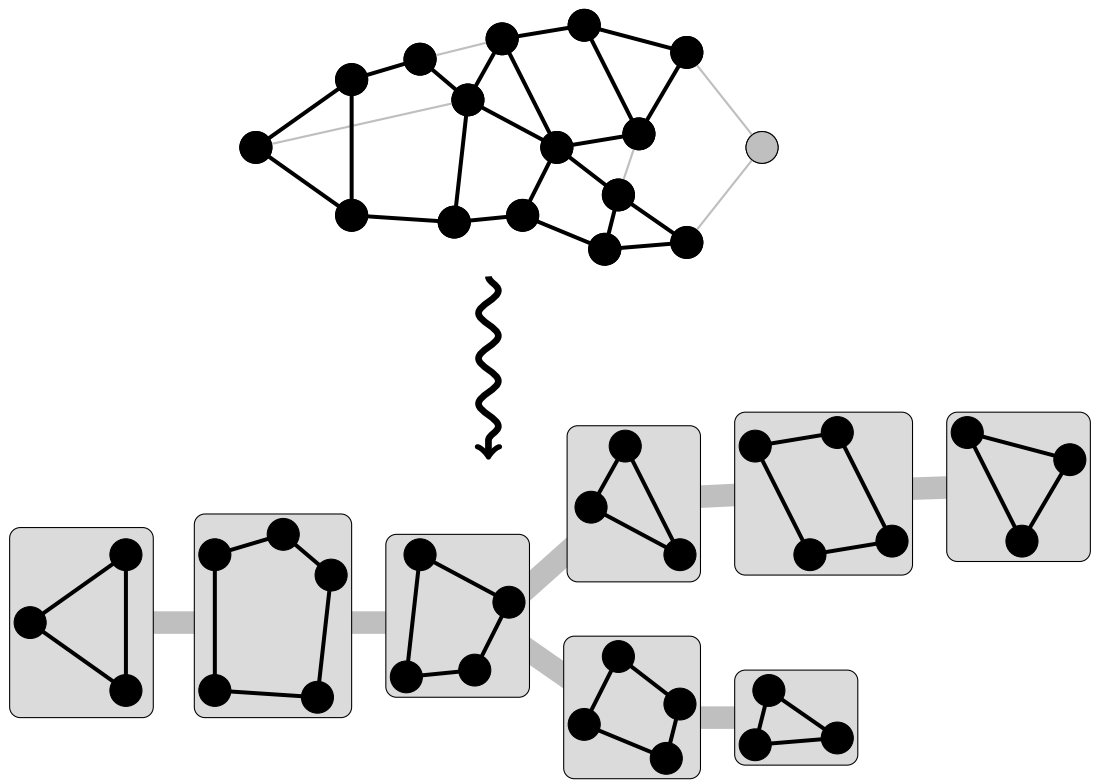
\begin{figure}
\centering
\begin{tikzpicture}[scale=0.9]
\node[circle,draw=black, fill=black, fill opacity = 1, inner sep=2pt, minimum size=12pt] (1v) at (-0.4,0+1) {};
\node[circle,draw=black, fill=black, fill opacity = 1, inner sep=2pt, minimum size=12pt] (2v) at (1,-1+1) {};
\node[circle, draw=black, fill=black, fill opacity = 1, inner sep=2pt, minimum size=12pt] (3v) at (1,1+1) {};
\node[circle,draw=black, fill=black, fill opacity = 1, inner sep=2pt, minimum size=12pt] (4v) at (2,1.3+1) {};
\node[circle,draw=black, fill=black, fill opacity = 1, inner sep=2pt, minimum size=12pt] (5v) at (2.5,-1.1+1) {};
\node[circle,draw=black, fill=black, fill opacity = 1, inner sep=2pt, minimum size=12pt] (6v) at (2.7,0.7+1) {};
\node[circle,draw=black, fill=black, fill opacity = 1, inner sep=2pt, minimum size=12pt] (7v) at (4,0+1) {};
\node[circle,draw=black, fill=black, fill opacity = 1, inner sep=2pt, minimum size=12pt] (8v) at (3.5,-1+1) {};
\node[circle,draw=black, fill=black, fill opacity = 1, inner sep=2pt, minimum size=12pt] (9v) at (3.2,1.6+1) {};
\node[circle,draw=black, fill=black, fill opacity = 1, inner sep=2pt, minimum size=12pt] (10v) at (5.2,0.2+1) {};
\node[circle,draw=black, fill=black, fill opacity = 1, inner sep=2pt, minimum size=12pt] (11v) at (4.4,1.8+1) {};
\node[circle,draw=black, fill=black, fill opacity = 1, inner sep=2pt, minimum size=12pt] (12v) at (5.9,1.4+1) {};
\node[circle,draw=black, fill=black, fill opacity = 1, inner sep=2pt, minimum size=12pt] (13v) at (4.7,-1.5+1) {};
\node[circle,draw=black, fill=black, fill opacity = 1, inner sep=2pt, minimum size=12pt] (14v) at (4.9,-0.7+1) {};
\node[circle,draw=black, fill=black, fill opacity = 1, inner sep=2pt, minimum size=12pt] (15v) at (5.9,-1.4+1) {};
\node[circle,draw=black, fill=lightgray, fill opacity = 1, inner sep=2pt, minimum size=12pt] (g1) at (7,0+1) {};

\node[circle,draw=none, fill=none, fill opacity = 0, inner sep=2pt, minimum size=12pt] (inv1) at (0.35-3.3,-6) {};
\node[circle,draw=none, fill=none, fill opacity = 0, inner sep=2pt, minimum size=12pt] (inv2) at (1.6-2,-6) {};
\node[circle,draw=none, fill=none, fill opacity = 0, inner sep=2pt, minimum size=12pt] (inv3) at (3,-6) {};
\node[circle,draw=none, fill=none, fill opacity = 0, inner sep=2pt, minimum size=12pt] (inv4) at (3+2,0.8-5) {};
\node[circle,draw=none, fill=none, fill opacity = 0, inner sep=2pt, minimum size=12pt] (inv5) at (4+3,0.9-5) {};
\node[circle,draw=none, fill=none, fill opacity = 0, inner sep=2pt, minimum size=12pt] (inv6) at (5,-7.4) {};
\node[circle,draw=none, fill=none, fill opacity = 0, inner sep=2pt, minimum size=12pt] (inv7) at (10,1-5) {};
\node[circle,draw=none, fill=none, fill opacity = 0, inner sep=2pt, minimum size=12pt] (inv8) at (8,-7.4) {};

\draw [-,    line width=8.5pt, lightgray, shorten <=-1pt, shorten >=-1pt] (inv1) to  (inv2);
\draw [-,    line width=8.5pt, lightgray, shorten <=-1pt, shorten >=-1pt] (inv2) to  (inv3);
\draw [-,    line width=8.5pt, lightgray, shorten <=-1pt, shorten >=-1pt] (inv3) to  (inv4);
\draw [-,    line width=8.5pt, lightgray, shorten <=-1pt, shorten >=-1pt] (inv4) to  (inv5);
\draw [-,    line width=8.5pt, lightgray, shorten <=-1pt, shorten >=-1pt] (inv3) to  (inv6);
\draw [-,    line width=8.5pt, lightgray, shorten <=-1pt, shorten >=-1pt] (inv5) to  (inv7);
\draw [-,    line width=8.5pt, lightgray, shorten <=-1pt, shorten >=-1pt] (inv6) to  (inv8);

\draw[rounded corners, fill={rgb:black,1;white,6}, fill opacity = 1] (-0.7-3.3, -1.4-6) rectangle (1.4-3.3, 1.4-6) {};
\node[circle,draw=black, fill=black, fill opacity = 1, inner sep=2pt, minimum size=12pt] (1w) at (-0.4-3.3,0-6) {};
\node[circle,draw=black, fill=black, fill opacity = 1, inner sep=2pt, minimum size=12pt] (2w) at (1-3.3,-1-6) {};
\node[circle, draw=black, fill=black, fill opacity = 1, inner sep=2pt, minimum size=12pt] (3w) at (1-3.3,1-6) {};

\draw[rounded corners, fill={rgb:black,1;white,6}, fill opacity = 1] (0.7-2, -1.4-6) rectangle (3-2, 1.6-6) {};
\node[circle,draw=black, fill=black, fill opacity = 1, inner sep=2pt, minimum size=12pt] (2p) at (1-2,-1-6) {};
\node[circle, draw=black, fill=black, fill opacity = 1, inner sep=2pt, minimum size=12pt] (3p) at (1-2,1-6) {};
\node[circle,draw=black, fill=black, fill opacity = 1, inner sep=2pt, minimum size=12pt] (4p) at (2-2, 1.3-6) {};
\node[circle,draw=black, fill=black, fill opacity = 1, inner sep=2pt, minimum size=12pt] (5p) at (2.5-2,-1.1-6) {};
\node[circle,draw=black, fill=black, fill opacity = 1, inner sep=2pt, minimum size=12pt] (6p) at (2.7-2,0.7-6) {};

\draw[rounded corners, fill={rgb:black,1;white,6}, fill opacity = 1] (2.2-0.7, -1.4-5.7) rectangle (4.3-0.7, 1-5.7) {};
\node[circle,draw=black, fill=black, fill opacity = 1, inner sep=2pt, minimum size=12pt] (5q) at (2.5-0.7,-1.1-5.7) {};
\node[circle,draw=black, fill=black, fill opacity = 1, inner sep=2pt, minimum size=12pt] (6q) at (2.7-0.7,0.7-5.7) {};
\node[circle,draw=black, fill=black, fill opacity = 1, inner sep=2pt, minimum size=12pt] (7q) at (4-0.7,0-5.7) {};
\node[circle,draw=black, fill=black, fill opacity = 1, inner sep=2pt, minimum size=12pt] (8q) at (3.5-0.7,-1-5.7) {};

\draw[rounded corners, fill={rgb:black,1;white,6}, fill opacity = 1] (2.15+2, -0.4-5) rectangle (4.1+2, 1.9-5) {};
\node[circle,draw=black, fill=black, fill opacity = 1, inner sep=2pt, minimum size=12pt] (7r) at (3.8+2,0-5) {};
\node[circle,draw=black, fill=black, fill opacity = 1, inner sep=2pt, minimum size=12pt] (6r) at (2.5+2,0.7-5) {};\node[circle,draw=black, fill=black, fill opacity = 1, inner sep=2pt, minimum size=12pt] (9r) at (3.0+2,1.6-5) {};

\draw[rounded corners, fill={rgb:black,1;white,6}, fill opacity = 1] (2.9+3.7, -0.3-5) rectangle (5.5+3.7, 2.1-5) {};
\node[circle,draw=black, fill=black, fill opacity = 1, inner sep=2pt, minimum size=12pt] (7s) at (4+3.7,0-5) {};\node[circle,draw=black, fill=black, fill opacity = 1, inner sep=2pt, minimum size=12pt] (9s) at (3.2+3.7,1.6-5) {};
\node[circle,draw=black, fill=black, fill opacity = 1, inner sep=2pt, minimum size=12pt] (10s) at (5.2+3.7, 0.2-5) {};
\node[circle,draw=black, fill=black, fill opacity = 1, inner sep=2pt, minimum size=12pt] (11s) at (4.4+3.7,1.8-5) {};

\draw[rounded corners, fill={rgb:black,1;white,6}, fill opacity = 1] (4.1+5.6, -0.1-5) rectangle (6.2+5.6, 2.1-5) {};
\node[circle,draw=black, fill=black, fill opacity = 1, inner sep=2pt, minimum size=12pt] (10t) at (5.2+5.6,0.2-5) {};
\node[circle,draw=black, fill=black, fill opacity = 1, inner sep=2pt, minimum size=12pt] (11t) at (4.4+5.6,1.8-5) {};
\node[circle,draw=black, fill=black, fill opacity = 1, inner sep=2pt, minimum size=12pt] (12t) at (5.9+5.6,1.4-5) {};

\draw[rounded corners, fill={rgb:black,1;white,6}, fill opacity = 1] (2.1+2, -1.8-6.5) rectangle (4.1+2, 0.3-6.5) {};
\node[circle,draw=black, fill=black, fill opacity = 1, inner sep=2pt, minimum size=12pt] (7u) at (2.9+2,0-6.5) {};
\node[circle,draw=black, fill=black, fill opacity = 1, inner sep=2pt, minimum size=12pt] (8u) at (2.4+2,-1-6.5) {};
\node[circle,draw=black, fill=black, fill opacity = 1, inner sep=2pt, minimum size=12pt] (13u) at (3.6+2,-1.5-6.5) {};
\node[circle,draw=black, fill=black, fill opacity = 1, inner sep=2pt, minimum size=12pt] (14u) at (3.8+2,-0.7-6.5) {};

\draw[rounded corners, fill={rgb:black,1;white,6}, fill opacity = 1] (4.4+2.2, -1.8-6.3) rectangle (6.2+2.2, -0.4-6.3) {};
\node[circle,draw=black, fill=black, fill opacity = 1, inner sep=2pt, minimum size=12pt] (13x) at (4.7+2.2,-1.5-6.3) {};
\node[circle,draw=black, fill=black, fill opacity = 1, inner sep=2pt, minimum size=12pt] (14x) at (4.9+2.2,-0.7-6.3) {};
\node[circle,draw=black, fill=black, fill opacity = 1, inner sep=2pt, minimum size=12pt] (15x) at (5.9+2.2,-1.4-6.3) {};


\draw [-,   thick,  lightgray, shorten <=-1pt, shorten >=-1pt] (g1) to  (15v);
\draw [-,   thick,  lightgray, shorten <=-1pt, shorten >=-1pt] (g1) to  (12v);
\draw [-,   thick,  lightgray, shorten <=-1pt, shorten >=-1pt] (14v) to  (10v);
\draw [-,   thick,  lightgray, shorten <=-1pt, shorten >=-1pt] (4v) to  (9v);
\draw [-,   thick,  lightgray, shorten <=-1pt, shorten >=-1pt] (6v) to  (1v);

\draw [-,    line width=1.5pt, black, shorten <=-1pt, shorten >=-1pt] (1v) to  (2v);
\draw [-,    line width=1.5pt, black, shorten <=-1pt, shorten >=-1pt] (1v) to  (3v);
\draw [-,    line width=1.5pt, black, shorten <=-1pt, shorten >=-1pt] (2v) to  (3v);
\draw [-,    line width=1.5pt, black, shorten <=-1pt, shorten >=-1pt] (3v) to  (4v);
\draw [-,    line width=1.5pt, black, shorten <=-1pt, shorten >=-1pt] (4v) to  (6v);
\draw [-,    line width=1.5pt, black, shorten <=-1pt, shorten >=-1pt] (6v) to  (5v);
\draw [-,    line width=1.5pt, black, shorten <=-1pt, shorten >=-1pt] (5v) to  (2v);
\draw [-,    line width=1.5pt, black, shorten <=-1pt, shorten >=-1pt] (5v) to  (8v);
\draw [-,    line width=1.5pt, black, shorten <=-1pt, shorten >=-1pt] (8v) to  (7v);
\draw [-,    line width=1.5pt, black, shorten <=-1pt, shorten >=-1pt] (7v) to  (6v);
\draw [-,    line width=1.5pt, black, shorten <=-1pt, shorten >=-1pt] (6v) to  (9v);
\draw [-,    line width=1.5pt, black, shorten <=-1pt, shorten >=-1pt] (7v) to  (9v);
\draw [-,    line width=1.5pt, black, shorten <=-1pt, shorten >=-1pt] (10v) to  (11v);
\draw [-,    line width=1.5pt, black, shorten <=-1pt, shorten >=-1pt] (7v) to  (10v);
\draw [-,    line width=1.5pt, black, shorten <=-1pt, shorten >=-1pt] (9v) to  (11v);
\draw [-,    line width=1.5pt, black, shorten <=-1pt, shorten >=-1pt] (12v) to  (10v);
\draw [-,    line width=1.5pt, black, shorten <=-1pt, shorten >=-1pt] (12v) to  (11v);
\draw [-,    line width=1.5pt, black, shorten <=-1pt, shorten >=-1pt] (15v) to  (14v);
\draw [-,    line width=1.5pt, black, shorten <=-1pt, shorten >=-1pt] (15v) to  (13v);
\draw [-,    line width=1.5pt, black, shorten <=-1pt, shorten >=-1pt] (13v) to  (14v);
\draw [-,    line width=1.5pt, black, shorten <=-1pt, shorten >=-1pt] (14v) to  (7v);
\draw [-,    line width=1.5pt, black, shorten <=-1pt, shorten >=-1pt] (13v) to  (8v);

\draw [-,    line width=1.5pt, black, shorten <=-1pt, shorten >=-1pt] (1w) to  (2w);
\draw [-,    line width=1.5pt, black, shorten <=-1pt, shorten >=-1pt] (1w) to  (3w);
\draw [-,    line width=1.5pt, black, shorten <=-1pt, shorten >=-1pt] (2w) to  (3w);

\draw [-,    line width=1.5pt, black, shorten <=-1pt, shorten >=-1pt] (2p) to  (3p);
\draw [-,    line width=1.5pt, black, shorten <=-1pt, shorten >=-1pt] (3p) to  (4p);
\draw [-,    line width=1.5pt, black, shorten <=-1pt, shorten >=-1pt] (4p) to  (6p);
\draw [-,    line width=1.5pt, black, shorten <=-1pt, shorten >=-1pt] (6p) to  (5p);
\draw [-,    line width=1.5pt, black, shorten <=-1pt, shorten >=-1pt] (5p) to  (2p);

\draw [-,   line width=1.5pt, black, shorten <=-1pt, shorten >=-1pt] (6q) to  (5q);
\draw [-,   line width=1.5pt, black, shorten <=-1pt, shorten >=-1pt] (5q) to  (8q);
\draw [-,    line width=1.5pt, black, shorten <=-1pt, shorten >=-1pt] (8q) to  (7q);
\draw [-,    line width=1.5pt, black, shorten <=-1pt, shorten >=-1pt] (7q) to  (6q);

\draw [-,   line width=1.5pt, black, shorten <=-1pt, shorten >=-1pt] (7r) to  (6r);
\draw [-,    line width=1.5pt, black, shorten <=-1pt, shorten >=-1pt] (6r) to  (9r);
\draw [-,   line width=1.5pt, black, shorten <=-1pt, shorten >=-1pt] (7r) to  (9r);

\draw [-,    line width=1.5pt, black, shorten <=-1pt, shorten >=-1pt] (7s) to  (9s);
\draw [-,    line width=1.5pt, black, shorten <=-1pt, shorten >=-1pt] (10s) to  (11s);
\draw [-,    line width=1.5pt, black, shorten <=-1pt, shorten >=-1pt] (7s) to  (10s);
\draw [-,    line width=1.5pt, black, shorten <=-1pt, shorten >=-1pt] (9s) to  (11s);

\draw [-,    line width=1.5pt, black, shorten <=-1pt, shorten >=-1pt] (10t) to  (11t);
\draw [-,    line width=1.5pt, black, shorten <=-1pt, shorten >=-1pt] (12t) to  (10t);
\draw [-,    line width=1.5pt, black, shorten <=-1pt, shorten >=-1pt] (12t) to  (11t);

\draw [-,    line width=1.5pt, black, shorten <=-1pt, shorten >=-1pt] (7u) to  (8u);
\draw [-,    line width=1.5pt, black, shorten <=-1pt, shorten >=-1pt] (13u) to  (14u);
\draw [-,    line width=1.5pt, black, shorten <=-1pt, shorten >=-1pt] (14u) to  (7u);
\draw [-,    line width=1.5pt, black, shorten <=-1pt, shorten >=-1pt] (13u) to  (8u);

\draw [-,    line width=1.5pt, black, shorten <=-1pt, shorten >=-1pt] (13x) to  (14x);
\draw [-,    line width=1.5pt, black, shorten <=-1pt, shorten >=-1pt] (15x) to  (14x);
\draw [-,    line width=1.5pt, black, shorten <=-1pt, shorten >=-1pt] (15x) to  (13x);

\node[circle,draw=black, fill=black, fill opacity = 1, inner sep=2pt, minimum size=12pt] (1v) at (-0.4,0+1) {};
\node[circle,draw=black, fill=black, fill opacity = 1, inner sep=2pt, minimum size=12pt] (2v) at (1,-1+1) {};
\node[circle, draw=black, fill=black, fill opacity = 1, inner sep=2pt, minimum size=12pt] (3v) at (1,1+1) {};
\node[circle,draw=black, fill=black, fill opacity = 1, inner sep=2pt, minimum size=12pt] (4v) at (2,1.3+1) {};
\node[circle,draw=black, fill=black, fill opacity = 1, inner sep=2pt, minimum size=12pt] (5v) at (2.5,-1.1+1) {};
\node[circle,draw=black, fill=black, fill opacity = 1, inner sep=2pt, minimum size=12pt] (6v) at (2.7,0.7+1) {};
\node[circle,draw=black, fill=black, fill opacity = 1, inner sep=2pt, minimum size=12pt] (7v) at (4,0+1) {};
\node[circle,draw=black, fill=black, fill opacity = 1, inner sep=2pt, minimum size=12pt] (8v) at (3.5,-1+1) {};
\node[circle,draw=black, fill=black, fill opacity = 1, inner sep=2pt, minimum size=12pt] (9v) at (3.2,1.6+1) {};
\node[circle,draw=black, fill=black, fill opacity = 1, inner sep=2pt, minimum size=12pt] (10v) at (5.2,0.2+1) {};
\node[circle,draw=black, fill=black, fill opacity = 1, inner sep=2pt, minimum size=12pt] (11v) at (4.4,1.8+1) {};
\node[circle,draw=black, fill=black, fill opacity = 1, inner sep=2pt, minimum size=12pt] (12v) at (5.9,1.4+1) {};
\node[circle,draw=black, fill=black, fill opacity = 1, inner sep=2pt, minimum size=12pt] (13v) at (4.7,-1.5+1) {};
\node[circle,draw=black, fill=black, fill opacity = 1, inner sep=2pt, minimum size=12pt] (14v) at (4.9,-0.7+1) {};
\node[circle,draw=black, fill=black, fill opacity = 1, inner sep=2pt, minimum size=12pt] (15v) at (5.9,-1.4+1) {};
\node[circle,draw=black, fill=lightgray, fill opacity = 1, inner sep=2pt, minimum size=12pt] (g1) at (7,0+1) {};

\draw [->,decorate,decoration={snake,amplitude= 1.3mm,segment length=6mm, post length=1.3mm}, line width=2.5pt, black] (3,-0.9) -- (3,-3.6);
\end{tikzpicture}
\caption{A tree is obtained from a set of cycle subgraphs. Any two of these cycles
have common edges if and only if they are adjacent in the tree, and if they do they have
a single common edge. In contrast, there is no restriction on the common vertices.
}
\label{fig:cycle_forest}
\end{figure}

\begin{theorem} [Cycle-Forest Bound for Network Dynamics] \label{thm:cycle_forest_bound}
Let~$c_1,\ldots,c_m$ be oriented cycles in~$G$ with no repeated edges.
Suppose that any two of these cycles have at most one common edge.
Let~$F$ be the graph with vertices~$c_1,\ldots,c_m$ and edges the pairs~$\{c_i,c_j\}$
such that~$c_i$ and~$c_j$ have common edges. Moreover, suppose that the graph~$F$ is a forest
with no isolated vertices.
Let~$L\subseteq F$ be obtained by choosing from each component of~$F$ all leaves except one.
Then the set of equilibria~$X$ of the dynamical system~\eqref{eq:network_dynamics_bis}
satisfies
\[
	\andim(X) \leq \dim_{\R} \homol_0(G) + \dim_{\R} \homol_1(G) - \abs{F} + \abs{L}.
\]
\end{theorem}

Before proving Theorem~\eqref{thm:cycle_forest_bound}, we discuss a concrete
example. Figure~\eqref{fig:cycle_forest} shows how a forest~$F$ (in this case a tree)
is obtained from certain cycles~$c_1,\ldots,c_8$ of a graph~$G$.
In this particular case we have~$\dim_{\R} \homol_0(G)=1$,
$\dim_{\R} \homol_1(G)=12$, $\abs{F}=8$, and~$\abs{L}=2$,
thus we obtain the upper bound~$7$.

\begin{proof} [Proof of Theorem~\eqref{thm:cycle_forest_bound}]
Throughout the proof we denote~$\homol_0 = \homol_0(G)$ and~$\homol_1=\homol_1(G)$.

By Proposition~\ref{prop:analytic_dimension}, it is enough to prove the upper bound
for the combinatorial dimension~$\cdim(X)$. Recall that the combinatorial
dimension of a set~$X\subseteq \R^V$ is not intrinsic: It depends on the ambient space~$\R^V$.
The idea of the proof is to replace the set of equilibria~$X$ in ``vertex space''
by a certain set~$Z$ in ``cycle space'', and to relate the combinatorial dimensions of~$X$ and~$Z$
in their corresponding ambient spaces.

By definition~$X=\{x\in \Gamma^V : Bf(B^t x) = \omega\}$.
If~$\omega\notin \im B$, then~$X$ is empty and the statement trivial.
Suppose that~$\omega = B\xi$ for some~$\xi\in \R^E$. Fix one such~$\xi$ for the rest of the proof.
A simple inductive argument shows that~$c_1,\ldots,c_m$ are linearly independent
as vectors in~$\R^E$.
Complete~$c_1,\ldots,c_m$ to a basis~$c_1,\ldots,c_{\abs{\homol_1}}$
of~$\homol_1\subseteq \R^E$ and let~$C = \{c_1,\ldots,c_{\abs{\homol_1}}\}$.
Consider the set
\[
	Z = \{z\in \R^C \mid
		z_{c_1} c_1 + \cdots z_{c_{\abs{\homol_1}}} c_{\abs{\homol_1}}
			+ \xi \in f(\im B^T) \}.
\]
Note that if~$z\in Z$ is given, there are at most finitely many~$y\in \im B^T$ such that
\[
	z_{c_1} c_1 + \cdots z_{c_{\abs{\homol_1}}} c_{\abs{\homol_1}} + \xi = f(y).
\]
Fix one such~$y$ for the rest of the proof.
Since~$y \in \im B^T$, there is~$x \in \Gamma^V$ such that~$B^T x = y$. There are infinitely
many such~$x$, since~$\dim_{\Gamma} \ker B^T = c(G) = \dim_{\R} \homol_0$.
However, there is only one such~$x$ once~$x_v$ is fixed for
a vertex~$v$ in each connected component.
Therefore
\[
	\cdim(X) \leq \dim_{\R} \homol_0 + \cdim(Z)
\]
where~$\cdim(X)$ is the combinatorial dimension of~$X$ as a subset of~$\Gamma^V$
and~$\cdim(Z)$ is the combinatorial dimension of~$Z$ as a subset of~$\R^C$.
It remains to bound~$\cdim(Z)$.

Let~$H$ be the graph with vertex set~$C=\{c_1,\ldots,c_{\abs{\homol_1}}\}$
and edges those pairs~$\{c_i,c_j\}$ such that~$c_i$ and~$c_j$ have common edges
as cycles in~$G$. By hypothesis~$F$ is an induced subgraph of~$H$.
We claim that~$Z$ is compatible with~$F$ relative to~$H\setminus F$. Note that
Lemma~\ref{lem:relative_induced} cannot be applied here since~$Z$
might not be compatible with~$H$.
Instead, we prove the claim directly using two identities. First, for every edge~${e_j}\in E$
we have
\begin{equation} \label{eq:hom_edges}
	z_{c_1} c_{1,{e_j}} + \cdots z_{c_{\abs{\homol_1}}} c_{\abs{\homol_1},{e_j}}
		+ \xi_{e_j} = f_{e_j}(y_{e_j}).
\end{equation}
Recall that~$c_{i,e_j}=+1,-1,0$ according to whether the cycle~$c_i$ traverses~$e_j$ according
to its orientation, in the opposite orientation, or does not contain~$e_j$.
Second, for every cycle~$c_i\in C$, since~$y\in \im B^T = \ker B^\perp$, we have
\begin{equation} \label{eq:hom_cycles}
	\sum_{e_j\in E} c_{i,e_j} y_{e_j} = 0.
\end{equation}
Fix~$z_{H\setminus F} = (z_{c_{m+1}},\ldots,z_{c_{\abs{\homol_1}}})$ arbitrarily.
Let~$c_p \in F$ and~$c_q\in N_F(c_p)$. We need to prove
\begin{equation} \label{eq:hom_determined}
	(z_{c_k})_{c_k \in N_F(c_p) \setminus \{c_q\} \cup \{c_p\}} \fin_{H\setminus F} z_{c_q}.
\end{equation}
By hypothesis the elements~$c_p$ and~$c_q$,
seen as cycles in~$G$, intersect in exactly one edge, say~$e_{pq}$.
From~\eqref{eq:hom_edges} with~$e_{j}=e_{pq}$ it follows that~$z_{c_q}$
is uniquely determined by~$z_{c_p}$, by~$y_{e_{pq}}$,
by~$z_{H\setminus F}$ and by~$\xi_{e_{pq}}$.
Therefore, it is enough to show that~$y_{e_{pq}}$
is determined by~$z_{N_F(c_p)\setminus \{c_q\} \cup \{c_p\}}$.
From~\eqref{eq:hom_cycles} with~$c_i=c_{p}$ it follows that~$y_{e_{pq}}$
is uniquely determined by~$(y_e)_{e\in c_p\setminus \{e_{pq}\}}$
where~$e\in c_p\setminus \{e_{pq}\}$ ranges among the edges traversed by the cycle~$c_p$
excluded the edge~$e_{pq}$. In turn, from~\eqref{eq:hom_edges} applied to
each~$e\in c_p\setminus \{e_{pq}\}$, since none of these~$e$ is an edge of~$c_q$,
it follows that~$(y_e)_{e\in c_p\setminus \{e_{pq}\}}$ is finitely determined
by~$z_{N_F(c_p)\setminus \{c_q\} \cup \{c_p\}}$.
This concludes the proof of~\eqref{eq:hom_determined}.

We proved that~$Z$ is compatible with~$F$
relative to~$H\setminus F$. Theorem~\ref{thm:pure_forest_bound} now implies
\[
	z_L \fin_{H\setminus F} z_{F\setminus L},
\]
which is the same as~$z_{L\cup (H\setminus F)}\fin z_{F\setminus L}$,
from which
\[
	\cdim(Z)\leq \abs{H} -\abs{F} + \abs{L},
\]
thus concluding the proof.
\end{proof}


\section{Acknowledgments}
We express our gratitude to Thomas Rot for suggesting
a proof of Proposition~\ref{prop:analytic_dimension} that avoids
the notion of covering dimension, and to Lies Beers and Raffaella Mulas for precious
discussions about Laplacian eigenvalues.


\bibliographystyle{siam}
\bibliography{refs}

\end{document}